\documentclass[12pt]{amsart}
\usepackage{amsfonts,amssymb,amscd,amsmath,enumerate,verbatim}
\usepackage[latin1]{inputenc}
\usepackage{amscd}
\usepackage{latexsym}
\usepackage{multicol}
\usepackage{graphicx}  
\usepackage{leftidx}
\usepackage{tikz}
\usepackage{tikz-cd}
\input xy

\xyoption{all}

\def\frk{\mathfrak}               

\def\mm{{\frk m}}

\def\Phi{{\frk N}}
\def\opn#1#2{\def#1{\operatorname{#2}}} 
\opn\chara{char} 
\opn\length{\ell} 
\opn\pd{pd} 
\opn\rk{rk}
\opn\projdim{proj\,dim} 
\opn\injdim{inj\,dim} 
\opn\rank{rank}
\opn\depth{depth} 
\opn\grade{grade} 
\opn\height{height}
\opn\embdim{emb\,dim} 
\opn\codim{codim}

\opn\Tr{Tr} 
\opn\bigrank{big\,rank}
\opn\superheight{superheight}
\opn\lcm{lcm}
\opn\trdeg{tr\,deg}
	\opn\reg{reg} 
	\opn\del{del} 
	\opn\ini{in} 
	\opn\Mon{Mon}
	\opn\size{size}
	\opn\mult{mult}
	\opn\lk{lk}
	\opn\cone{cone}
	\opn\lex{lex}
	\opn\rev{rev}
	\opn\div{div} \opn\Div{Div} \opn\cl{cl} \opn\Cl{Cl}
	\opn\Spec{Spec} \opn\Supp{Supp} \opn\supp{supp} \opn\m{m}
	\opn\Ass{Ass} \opn\Min{Min}
	\opn\Ann{Ann} \opn\Rad{Rad} \opn\Soc{Soc}
	\opn\Syz{Syz} \opn\Im{Im} \opn\Ker{Ker} \opn\Coker{Coker}
	\opn\Am{Am} \opn\Hom{Hom} \opn\Tor{Tor} \opn\Ext{Ext}
	\opn\End{End} \opn\Aut{Aut} \opn\id{id} \opn\ini{in}
	
	\opn\nat{nat}
	\opn\pff{pf}
	\opn\Pf{Pf} \opn\GL{GL} \opn\SL{SL} \opn\mod{mod} \opn\ord{ord}
	\opn\Gin{Gin}
	\opn\Hilb{Hilb}\opn\adeg{adeg}\opn\std{std}\opn\ip{infpt}
	\opn\Pol{Pol}
	\opn\sat{sat}
	\opn\Var{Var}
	\opn\Gen{Gen}
	
	\opn\aff{aff} \opn\con{conv} \opn\relint{relint} \opn\st{st}
	\opn\lk{lk} \opn\cn{cn} \opn\core{core} \opn\vol{vol}
	\opn\link{link} \opn\star{star}
	\opn\gr{gr}
	
	\def\pot#1#2{#1[\kern-0.28ex[#2]\kern-0.28ex]}

	\opn\dirlim{\underrightarrow{\lim}}
	\opn\inivlim{\underleftarrow{\lim}}
	\let\Dirsum=\bigoplus
	
	\let\to=\rightarrow
	
	\def\Implies{\ifmmode\Longrightarrow \else
		\unskip${}\Longrightarrow{}$\ignorespaces\fi}
	\def\implies{\ifmmode\Rightarrow \else
		\unskip${}\Rightarrow{}$\ignorespaces\fi}
	\def\iff{\ifmmode\Longleftrightarrow \else
		\unskip${}\Longleftrightarrow{}$\ignorespaces\fi}

	\let\:=\colon
	\newtheorem{Theorem}{Theorem}[section]
	
	\newtheorem{Corollary}[Theorem]{Corollary}
	\newtheorem{Proposition}[Theorem]{Proposition}

	\let\epsilon\varepsilon
	\let\phi=\varphi
	\let\kappa=\varkappa
	\def\qed{\ifhmode\textqed\fi
		\ifmmode\ifinner\quad\qedsymbol\else\dispqed\fi\fi}
	\def\textqed{\unskip\nobreak\penalty50
		\hskip2em\hbox{}\nobreak\hfil\qedsymbol
		\parfillskip=0pt \finalhyphendemerits=0}
	\def\dispqed{\rlap{\qquad\qedsymbol}}
	
	\opn\SMon{SMon}
	\opn\height{height}
	\opn\dist{dist}
	\opn\supp{supp}
	\def\pnt{{\raise0.5mm\hbox{\large\bf.}}}
	
	\opn\Lex{Lex}
	
\begin{document}
		
\title{Normal Rees algebras arising from vertex decomposable simplicial complexes}

\author{Somayeh Moradi}

\address{Somayeh Moradi, Department of Mathematics, Faculty of Science, Ilam University,
	P.O.Box 69315-516, Ilam, Iran}
\email{so.moradi@ilam.ac.ir}

\dedicatory{ }
\thanks{ The author is supported by the Alexander von Humboldt Foundation.}

\begin{abstract}
We show that for a vertex decomposable simplicial complex $\Delta$, the Rees algebra of $I_{\Delta^{\vee}}$ is a normal Cohen-Macaulay domain. As consequences, we show that any squarefree weakly polymatroidal ideal is normal and we obtain normal ideals among several interesting families of monomial ideals such as  cover ideals of graphs and edge ideals of hypergraphs. Moreover, based on a construction on simplicial complexes given by Biermann and Van Tuyl~\cite{BV}, we present families of normal ideals attached to any squarefree monomial ideal.   
\end{abstract}

\maketitle

\setcounter{tocdepth}{1}
\section{Introduction}		

Let $S=K[x_1,\ldots,x_n]$ be a polynomial ring over a field $K$, and let $I\subset S$ be a nonzero monomial ideal. The {\em Rees algebra} of $I$ is the subring of the polynomial ring $S[t]$ given by $\mathcal{R}(I)=\Dirsum_{s\geq 0} I^st^s$. Rees algebras play an important role in the study of powers of ideals, namely on the behaviour of their depth, their associated prime ideals, and properties like having linear resolutions or componentwise linear resolutions.  One of the major aspects in the study of Rees algebras is to determine when they are normal, since by a theorem of Hochster~\cite{Ho} this implies that $\mathcal{R}(I)$ is Cohen-Macaulay.  
It is of particular interest to study the Rees algebra of $I$, when  $I$ is a monomial ideal which arises from a combinatorial object. This allows to determine the normality of $\mathcal{R}(I)$ in terms of the attached combinatorial object. When $I=I(G)$ is the edge ideal of a finite simple graph $G$, then the Rees algebra $\mathcal{R}(I)$ is just the edge ring of the graph $C(G)$, where $C(G)$ is the cone over $G$. Hence, using the characterization of normal edge rings given by Hibi and Ohsugi~\cite{HO}, and independently by  Simis, Vasconcelos and Villarreal~\cite{SWV2}, the normality of $\mathcal{R}(I)$ can be expressed by combinatorial properties of the graph $G$,  see~\cite[Theorem 2.2]{HH3}. For cover ideals of graphs there is no such characterizations and much less is known. The Rees algebra of the cover ideal of a graph is shown to be normal when $G$ is a perfect graph~\cite{Vi}, a  Cohen-Macaulay very well-covered graph~\cite{CF}, or an odd cycle~\cite{A}. In~\cite{D}, for the cover ideal of a graph $G$ with  independence number at most $2$ a criterion for the normality is given in terms of Hochster configurations.
Polymatroidal ideals are another family of ideals arising from combinatorics for which the normality is known~\cite{Vi3}. 
  
In this paper we use a criteria for normality given in~\cite{NQBM} and show the normality of $\mathcal{R}(I)$, when $I$ is a squarefree monomial ideal which appears as the Alexander dual of the Stanley-Reisner ideal of a vertex decomposable simplicial complex (Theorem~\ref{main}). Such ideals, known as vertex splittable ideals, were  studied in~\cite{MK}. When $\Delta$ is the independence complex of a graph $G$, the above ideal $I$ is just the cover ideal of $G$. This implies that cover ideals of vertex decomposable graphs are normal. Among such graphs are graphs with no induced cycles of
length other than $3$ or $5$,  Cohen-Macaulay cactus graphs, Cameron-Walker graphs,  Cohen-Macaulay very well-covered graphs, fully clique-whiskered graphs and   Cohen-Macaulay graphs of girth at least $5$, see Corollary~\ref{Anschreiben}. 
Another application of Theorem~\ref{main} is given in Proposition~\ref{defenceAntonino}, which shows the normality and the strong persistence property for squarefree weakly polymatroidal ideals. 
A result of
Lu and Wang~\cite{LW} is the crucial tool to obtain this fact. Proposition~\ref{defenceAntonino} together with results of Rahmati-Asghar and Yassemi~\cite{RY} which proved weakly polymatroidal property for edge  ideals of some hypergraphs, allows us to present some hypergraphs with normal edge ideals,
see Corollary~\ref{havayekhoob}. 
Applying our main theorem to a class of monomial ideals which appeared first in~\cite{BHHM}, in Corollary~\ref{printer} we see that for a monomial ideal $I$ with $\dim S/I=0$, the ideal
$(I^{\wp})^{\vee}$ is normal, where $I^{\wp}$ denotes the polarization of $I$.  For the edge ideal of a cochordal graph $G$, in Proposition~\ref{berenj} the limit depth  and bounds for the regularities of edge and Rees rings (based on results in~\cite{HH2} and \cite{HH3}) are given. Finally, we use the a construction on simplicial complexes given in \cite{BV} to construct normal ideals from a given squarefree monomial ideal, see Proposition~\ref{eshgh}.


\section{Preliminaries}
Throughout this paper $S=K[x_1,\ldots,x_n]$ is a polynomial ring over a field $K$. We denote by $\Mon(S)$ the set of monomials in $S$, and by $G(I)$ the minimal set of monomial generators of a monomial ideal $I$. The family of vertex splittable ideals were defined in \cite{MK} as follows:

A monomial ideal $I\subset S$ is called  {\em vertex splittable}  if it can be obtained by the following recursive procedure.
\begin{itemize}
	\item[(i)] If $I=(0)$ or $I=S$, then $I$ is a vertex splittable ideal.
	\item[(ii)] If there exists an integer $1\leq i\leq n$ and vertex splittable ideals $I_1$ and $I_2$ of $K[x_j:\  j\in [n]\setminus \{i\}]$ so that $I=x_iI_1+I_2$ with $I_2\subseteq I_1$ and $G(I)$ is the disjoint union of $G(x_iI_1)$ and $G(I_2)$, then $I$ is a vertex splittable ideal.
\end{itemize}

The decomposition  $I=x_iI_1+I_2$ is called a {\em vertex splitting} for $I$.
Vertex splitting is the dual concept of vertex decomposability. In fact in~\cite[Theorem 2.3]{MK}, it is shown that $\Delta$ is a vertex decomposable simplicial complex if and only if $I_{\Delta^{\vee}}$ is vertex splittable.
We recall the definition of a vertex decomposable simplicial complex.
For a simplicial complex $\Delta$ and $F\in \Delta$, the {\em link} of $F$ in
$\Delta$ is defined as $$\lk_{\Delta}(F)=\{G\in \Delta: G\cap
F=\emptyset, G\cup F\in \Delta\},$$ and the {\em deletion} of $F$ is the
simplicial complex $\del_{\Delta}(F)=\{G\in \Delta: G\cap
F=\emptyset\}.$
A simplicial complex $\Delta$ is  called {\em vertex decomposable} if $\Delta=\emptyset$, or
$\Delta$ is a simplex, or there exists a vertex $x\in V(\Delta)$ such that
both $\del_{\Delta}(x)$ and $\lk_{\Delta}(x)$ are vertex decomposable, and
any facet of $\del_{\Delta}(x)$ is a facet of $\Delta$. Such a vertex $x$ is called a {\em shedding vertex} of $\Delta$. 

For a finite simple graph $G$, we denote the vertex set and the edge set of $G$ by $V(G)$ and $E(G)$, respectively.   A subset $C\subseteq V(G)$ is called a \emph{vertex cover} of $G$,
if it intersects any edge of $G$, and it is called a \emph{minimal vertex cover} of $G$, if it is a vertex cover of $G$ and no proper subset of $C$ has this property.
Let $V(G)=[n]$ and let $C_1,\ldots,C_m$  be the minimal vertex covers of $G$. The {\em   cover ideal}  of $G$ is a squarefree monomial ideal in the polynomial ring $S=K[x_1,\ldots,x_n]$ defined as
$J(G)=(x_{C_1},\ldots,x_{C_m}),$ where $x_{C_j}=\prod_{i\in C_j} x_i$. 

A {\em matching} of the graph $G$ is a set  of pairwise disjoint edges in $G$. The maximum cardinality of a matching of $G$ is called the {\em matching number} of $G$, denoted by $\m(G)$. 
 
A subset $A\subseteq V(G)$ is called an {\em independent set} of $G$ if it contains no edge of $G$. The {\em independence complex} $\Delta_G$ of the graph $G$ is a simplicial complex with the vertex set $V(G)$ whose faces are independent sets of $G$.  
A graph $G$ is called {\em vertex decomposable} if $\Delta_G$ is vertex decomposable.


The {\em support} of a monomial $u\in S$ is the set $\supp(u)=\{x_i: x_i\textrm{ divides } u\}$. For a monomial ideal $I$ with $G(I)=\{u_1,\ldots,u_m\}$, we set $K[I]=K[u_1,\ldots,u_m]$. The analytic spread of $I$ is defined as the dimension of the fiber ring $\mathcal{R}(I)/\mm \mathcal{R}(I)$ and is denoted by $\ell(I)$. 

\section{Vertex decomposable simplicial complexes and normality}

In this section we present families of normal monomial ideals arising from combinatorics.  
		
\begin{Theorem}\label{main}		Let $\Delta$ be a vertex decomposable simplicial complex on $[n]$, and let $I=I_{\Delta^{\vee}}$. Then 
\begin{enumerate}
		\item [{\em(a)}] The Rees algebra $\mathcal{R}(I)$ is a normal Cohen-Macaulay domain.
		\item [{\em(b)}] The ideal $I$ satisfies the strong persistence property, and in particular, it satisfies the persistence property. 
		\item [{\em(c)}] $\lim_{k\to\infty} \depth S/I^k=n-\ell(I)$. 
		\item [{\em(d)}] If in addition $\Delta$ is pure, then the toric ring $K[I]$ is  normal and Cohen-Macaulay.
\end{enumerate}
\end{Theorem}

\begin{proof}
(a) By~\cite[Proposition 2.1.2]{HSV}, the normality of $\mathcal{R}(I)$ follows once we  prove that $I=I_{\Delta^{\vee}}$ is a normal ideal. If  $I=(0)$ or $I=S$, there is nothing to prove. Let $(0)\neq I\subsetneq S$.  We show the assertion by induction on the cardinality of $V(\Delta)$. Since $I$ is a nonzero proper ideal of $S$, $\Delta\neq \emptyset$ and $\Delta$ is not a simplex. So it has a shedding vertex, say $x_i$. Let $\Delta_1=\del_{\Delta}(x_i)$ and $\Delta_2=\lk_{\Delta}(x_i)$. Then by \cite[Theorem 2.3]{MK}, $I=x_iI_{\Delta_1^{\vee}}+I_{\Delta_2^{\vee}}$ is a vertex splitting of $I$. The complexes $\Delta_1$ and $\Delta_2$ are vertex decomposable. So by induction hypothesis we may assume that $I_{\Delta_1^{\vee}}$ and $I_{\Delta_2^{\vee}}$ are normal. Since $I_{\Delta_2^{\vee}}\subseteq I_{\Delta_1^{\vee}}$, it follows from~\cite[Theorem 3.1]{NQBM} (see also ~\cite[Criterion 2.2]{CF}) that $I$ is normal, as well. 
Now we apply the result of Hochster \cite[Theorem 1]{Ho} which shows that any normal affine semigroup ring is Cohen-Macaulay. 

(b) follows from (a) and \cite[Corollary 1.6]{HQ}.

(c)  Part (a) together with \cite[Proposition 1.1]{Hu} imply that the  associated graded ring
$\gr_I(S)$ is Cohen-Macaulay. Hence by the result of Eisenbud and Huneke~\cite[Proposition 3.3]{EH} the desired equality holds. 

(d)  Since $\Delta$ is pure, $I$ is generated in one degree. So the result follows from (a) and \cite[Theorem 7.1]{SWV}.
\end{proof}

The normality and the persistence property of polymatroidal ideals were proved in \cite[Proposition 3.11]{Vi3} and  \cite[Proposition 2.3, Theorem 2.4]{HRV}:

\begin{Theorem}\label{meh}
Let $I\subset S$ be a polymatroidal ideal. Then $I$ has the persistence property and $\mathcal{R}(I)$ is a normal ring.
\end{Theorem}

As generalization of polymatroidal ideals, weakly polymatroidal ideals were first introduced in \cite{KH} for equigenerated monomial ideals and later studied in \cite{MM} for monomial ideals  generated in arbitrary degrees.   
A monomial ideal $I\subset S$ is called {\em weakly
polymatroidal} if for every pair of monomials $u=x_1^{a_1}\cdots x_n^{a_n}$ and $v=x_1^{b_1}\cdots x_n^{b_n}$ in $G(I)$
with $a_1 = b_1,\ldots,a_{i-1} =b_{i-1}$ and $a_i<b_i$, there exists $j>i$
such that $x_i(u/x_j)$ belongs to $G(I)$.

The following result generalizes Theorem \ref{meh} partially.  

\begin{Proposition}\label{defenceAntonino}
Let $I$ be a squarefree weakly polymatroidal  ideal. Then $\mathcal{R}(I)$ is a normal Cohen-Macaulay domain and $I$ satisfies the strong persistence property. 
\end{Proposition}

\begin{proof}
Let $I$ be a squarefree weakly polymatroidal ideal, and let $\Delta=\Delta_I$ be the Stanley-Reisner simplicial complex of $I$. 
Then \cite[Theorem 3.1]{LW} implies that $\Delta^{\vee}$ is vertex decomposable. Hence  $I_{(\Delta^{\vee})^{\vee}}=I$ satisfies the properties of (a) to (d) in Theorem~\ref{main}. 
\end{proof}

In \cite{RY}, Rahmati-Asghar and Yassemi studied the weakly polymatroidal property for ideals attached to  hypergraphs and obtained some hypergraphs whose edge ideals are weakly polymatroidal.  Applying Proposition~\ref{defenceAntonino} to these ideals we get families of normal ideals. For the concepts related to hypergraphs used in Corollary~\ref{havayekhoob} we refer the reader to \cite{RY}.


\begin{Corollary}\label{havayekhoob}
 Let $\mathcal{H}$ be one of the following hypergraphs:
 
\begin{itemize}
	\item A $d$-partite Ferrers $d$-uniform hypergraph without isolated vertices. 
	\item Complement of a chordal hypergraph. 
	\item  A complete admissible uniform hypergraph or complement of such a hypergraph.
\end{itemize}
Let $I=I(\mathcal{H})$ be the edge ideal of $\mathcal{H}$. Then 
$\mathcal{R}(I)$ is a normal Cohen-Macaulay domain and $I$ satisfies the strong persistence property.  
\end{Corollary}

\begin{proof}
By \cite[Theorem 2.7, Theorem 3.2, Theorem 4.1]{RY} the ideal $I(\mathcal{H})$ is weakly polymatroidal. Thus the result follows from Proposition~\ref{defenceAntonino}.  
\end{proof}

A family of monomial ideals whose normality can be obtained by Theorem~\ref{main} appeared in the work of Bigdeli, Herzog, Hibi, and Macchia on whisker type simplicial complexes \cite{BHHM}. For a monomial ideal $I\subset S$, let $I^{\wp}$ denote the polarization of $I$, and let $S^{\wp}$ be the polynomial ring in which $I^{\wp}$ is defined. Attached to a monomial ideal $I$ with $\dim S/I=0$, the simplicial complex
$\Delta(I)$ is defined to be the Stanley-Reisner simplicial complex of $I^{\wp}$.  
The ideal   $L(I)=I_{\Delta(I)^{\vee}}$, which is indeed $(I^{\wp})^{\vee}$, is generated by the monomials $x_{1,a_1+1}\cdots x_{n,a_n+1}$ with $x_1^{a_1}\cdots x_n^{a_n}\in \Mon(S)\setminus I$, see~\cite[Corollary 1.2]{BHHM}. 
The following result shows that $L(I)$ is normal.
  
\begin{Corollary}\label{printer}
Let $I\subset S$ be a monomial ideal  with $\dim S/I=0$. Then $\mathcal{R}(L(I))$ is a normal Cohen-Macaulay domain and $L(I)$ satisfies the strong persistence property. 
\end{Corollary}

\begin{proof}
By \cite[Theorem 3.1]{BHHM},  $\Delta(I)$ is vertex decomposable. So the result follows from Theorem~\ref{main}.   
\end{proof}

Next, applying Theorem \ref{main} to cover ideals of vertex decomposable graphs, we present some families of  graphs whose cover ideals are normal. For the graph concepts used in Corollary~\ref{Anschreiben}, we refer to the related references mentioned in its proof.

	\begin{Corollary}\label{Anschreiben}
	Let $G$ be a vertex decomposable graph,	and let $J=J(G)$ be the cover ideal of $G$. Then the Rees algebra $\mathcal{R}(J)$ is a normal Cohen-Macaulay domain and $J$ satisfies the strong persistence property. In particular, this holds true if $G$ is one of the following graphs:
	\begin{enumerate}
		\item[\em{(i)}] A chordal graph (see \cite[Theorem 2.10]{Vi}).
		\item[\em{(ii)}] A Cohen-Macaulay very well-covered graph (see ~\cite[Corollary 3.5]{CF}). 
		\item[\em{(iii)}] A Cameron-Walker graph.
		\item[\em{(iv)}] A graph with no induced cycles of
		length other than $3$ or $5$.
		\item[\em{(v)}] A Cohen-Macaulay cactus graph.
		\item[\em{(vi)}] A fully clique-whiskered graph.
		\item[\em{(vii)}] A graph $G$ which is constructed by attaching a complete graph $K_{m_i}$ to each vertex $i$ of an arbitrary graph $H$.
		\item[\em{(viii)}] A Cohen-Macaulay graph of girth at least $5$.   
	\end{enumerate}  
\end{Corollary}

\begin{proof}
The first statement follows from Theorem~\ref{main} and the fact that $J(G)=I_{\Delta_G^{\vee}}$. 
For each graph $G$ belonging to one of the mentioned families, $\Delta_G$ is vertex decomposable, see \cite[Corollary 7]{W}, \cite[Theorem 3.2]{MMCRTY}, \cite[Theorem 3.1]{HHKO}, \cite[Theorem 1]{W}, \cite[Theorem 2.8]{Mo}, \cite[Theorem 3.3]{CN}, \cite[Theorem 1.1]{HHKO} and \cite[Theorem 20]{BC}, respectively.  
\end{proof}

It should be mentioned that \cite[Theorem 2.10]{Vi} shows the normality of $J(G)$ for the class of perfect graphs, which includes the class of chordal graphs. 
\medskip 

Applying Theorem~\ref{main} to edge ideals of graphs we have

\begin{Proposition}\label{berenj}
	Let $G$ be a cochordal graph, let $I=I(G)\neq (0)$ be the edge ideal of $G$, and let $s$ be the number of isolated vertices of $G$. Then
	\begin{itemize}
	\item [(i)] $\mathcal{R}(I)$ and $K[I]$ are normal. 
	\item [(ii)]  $\lim_{k\to\infty} \depth S/I^k=\begin{cases}
		s+1, & \text{if $G$ is bipartite}, \\
	s, & \text{otherwise}.
	\end{cases}$
		\item [(iii)] $\reg K[I]\leq \m(G)$.
		\item [(iv)] $\m(G)\leq \reg \mathcal{R}(I)\leq \m(G)+1$.
	\end{itemize}
\end{Proposition}

\begin{proof}
(i) By \cite[Corollary 3.8]{MK}, $(\Delta_G)^{\vee}$ is{\tiny } vertex decomposable if and only if $G$ is a cochordal graph. Thus the result follows from Theorem~\ref{main}, noting that $(\Delta_G)^{\vee}$ is pure.

(ii)  The fiber ring $\mathcal{R}(I)/\mm\mathcal{R}(I)$ is isomorphic to the toric ring $K[I]$. So $\ell(I)$ is equal to the rank of the incidence matrix $A_G$ of $G$. Since $G$ is cochordal and $I(G)$ is nonzero, the induced matching number of $G$ is one. This implies that $G$ has exactly one connected component which is not an isolated vertex, say $G_1$. Let $n=|V(G)|$. By \cite[Theorem 2.5]{GK}, $\rank A_G=n-s-1$, if $G$ (or equivalently $G_1$) is bipartite and  $\rank A_G=n-s$, otherwise. Now, by Theorem~\ref{main}(d) the formula for the limit depth follows.  

(iii) By (i), $K[I(G_1)]$ is normal. Since $G_1$ is connected, this together with \cite[Theorem 1]{HH2} gives $\reg K[I(G_1)]\leq \m(G_1)$. Noting that $K[I]=K[I(G_1)]$ and $\m(G)=\m(G_1)$, the desired inequality holds. 

(iv) follows from (i) and \cite[Theorem 2.2]{HH3}.     
\end{proof}

Biermann and Van Tuyl in \cite[Construction 3]{BV} presented a construction on a simplicial complex $\Delta$ using the notion of a colouring $\chi$ of $\Delta$ which always results in a balanced vertex decomposable simplicial complex. In view of Theorem~\ref{main} and \cite[Theorem 5]{BV}, by translating this construction  to squarefree monomial ideals, one can construct normal squarefree monomial ideals attached to a given squarefree monomial $I$ and a given colouring of $I$, as follows:

Let $I\subset S$ be a squarefree monomial ideal, and let $\SMon(I)=\{u_1,\ldots,u_m\}$ be the set of all squarefree monomials in $I$. A {\em colouring} of $I$ is a partition $\bigcup_{i=1}^s V_i$ of $X=\{x_1,\ldots,x_n\}$ such that $|V_i\setminus \supp(u_j)|\leq 1$ for all $i$ and $j$. For example taking $V_i=\{x_i\}$ for all $i$ one gets a colouring of $I$. 
Let $\chi$ be a colouring of $I$ given by the partition  $\bigcup_{i=1}^s V_i$. Attached to $I$ and $\chi$, the squarefree monomial ideal $I_{\chi}$ is defined to be an ideal in the polynomial ring $T=K[x_1,\ldots,x_n,y_1,\ldots,y_s]$ as follows:  
$$I_{\chi}=(u_iw_i:\ 1\leq i\leq m ),$$
where $w_i=\prod\limits_{V_j\nsubseteq \supp(u_i)} y_j$. 

For example, let $I=(x_1x_2,x_2x_4,x_1x_3x_4)\subset K[x_1,x_2,x_3,x_4]$, and let  $\chi$ be a colouring of $I$ given by the partition $V_i=\{x_i\}$ for all $i$. Then $$\SMon(I)=\{x_1x_2,x_1x_2x_3,x_1x_2x_4,x_2x_3x_4,x_1x_3x_4,x_1x_2x_3x_4\},$$  
and 
$$I_{\chi}=(x_1x_2y_3y_4,x_1x_2x_3y_4,x_1x_2x_4y_3,x_2x_3x_4y_1,x_1x_3x_4y_2,x_1x_2x_3x_4).$$


\begin{Proposition}\label{eshgh}
Let $I\subset S=K[x_1,\ldots,x_n]$ be a squarefree monomial ideal, and let  $\chi$ be a colouring of $I$. Then
$I_{\chi}$ is a vertex splittable ideal generated in degree $n$. In particular, $\mathcal{R}(I_{\chi})$ and $K[I_{\chi}]$ are normal Cohen-Macaulay domains.
\end{Proposition}

\begin{proof}
Let $\Delta$ be the Stanley-Reisner simplicial complex of $I^{\vee}$, and let $\SMon(I)=\{u_1,\ldots,u_m\}$.  Set $X=\{x_1,\ldots,x_n\}$, $Y=\{y_1,\ldots,y_s\}$ and $C_i=\supp(u_i)$ for all $i$. Then $\Delta=\{X\setminus C_1,\ldots,X\setminus C_m\}$. Let  $V_1\cup\cdots\cup V_s$ be the partition defined by the colouring $\chi$, and set $F_j=X\setminus C_j$ for all $j$. Then $|V_i\cap F_j|=|V_i\setminus \supp(u_j)|\leq 1$ for all $i$ and $j$. Therefore
$\chi$ is an $s$-colouring of $\Delta$ (see \cite[Definition 2]{BV}). So by \cite[Theorem 5, Theorem 7]{BV}, the simplicial complex $\Delta_{\chi}=\langle F\cup\{y_j: V_j\cap F=\emptyset\}: F\in\Delta\rangle$ is pure vertex decomposable of dimension $s-1$. Applying \cite[Theorem 2.3]{MK} we conclude that $I_{\Delta_{\chi}^{\vee}}$ is vertex splittable. We show that $I_{\Delta_{\chi}^{\vee}}=I_{\chi}$. 
For subsets $A\subseteq [n]$ and $B\subseteq [s]$, we set $x_Ay_B=\prod\limits_{i\in A}x_i\prod\limits_{j\in B}y_j$, $X_A=\{x_i: i\in A\}$ and $Y_B=\{y_j: j\in B\}$. One has $x_Ay_B\in G(I_{\Delta_{\chi}^{\vee}})$ if and only if $X_{A^{c}}\cup Y_{B^{c}}$ is a facet of $\Delta_{\chi}$.  
This is equivalent to $X_{A^{c}}\in \Delta$ and  $B^c=\{j: V_j\cap X_{A^{c}}=\emptyset\}$.    This means that $X_A=C_i$ for some $i$ and $B=\{j: V_j\nsubseteq C_i\}$. Hence 
$x_Ay_B\in G(I_{\Delta_{\chi}^{\vee}})$ if and only if $x_Ay_B\in G(I_{\chi})$. 
Now, by Theorem~\ref{main} and the equality $I_{\Delta_{\chi}^{\vee}}=I_{\chi}$ we conclude that $\mathcal{R}(I_{\chi})$ and $K[I_{\chi}]$ are normal. 
\end{proof}   




\end{document}